\documentclass[11pt,reqno, english]{amsart}
\usepackage[a4paper]{geometry}
\usepackage{amsmath,amssymb,amscd,amsthm,amsfonts}
\usepackage{graphicx,subcaption}
\usepackage{hyperref}
\usepackage{dsfont}
\usepackage{xcolor}
\usepackage{color,soul}
\usepackage[utf8]{inputenc}
\usepackage[nobysame, alphabetic]{amsrefs}
\usepackage{tikz}
\usepackage[capitalise]{cleveref}

\newtheorem{theorem}{Theorem}[section]
\newtheorem{corollary}[theorem]{Corollary}
\newtheorem{lemma}[theorem]{Lemma}

\theoremstyle{definition}

\newtheorem{definition}[theorem]{Definition}

\newtheorem{claim}[theorem]{Claim}

\newcommand{\R}{\mathds{R}}
\newcommand{\rr}{\mathds{R}}
\newcommand{\zz}{\mathds{Z}}

\DeclareMathOperator{\Conv}{conv}

\title[Partitions of mass assignments with spheres and wedges]{Partitions of mass assignments with spheres and wedges}

\author[Lessure]{Deron Lessure}\address{University of Michigan, Ann Arbor, MI 48109} 
\email{dlessure@umich.edu}

\author[Sober\'on]{Pablo Sober\'on}\address{Baruch College \& The Graduate Center, City University of New York, New York, NY 10010} 
\email{psoberon@gc.cuny.edu}



\thanks{The research of P. Sober\'on was supported by NSF CAREER award no. 2237324, NSF award no. 2054419 and a PSC-CUNY Trad B award.}

\begin{document}

\begin{abstract}
In this paper, we generalize classic mass partition results dealing with partitions using spheres, parallel hyperplanes, or axis-parallel wedges to the setting of mass assignments.  In a mass assignment problem, we assign mass distributions continuously to all $k$-dimensional subspaces of $\mathbb{R}^d$, and seek to guarantee the existence of a particular subspace in which more masses can be bisected than those by analyzing the problem in $\rr^k$. We prove new mass assignment results for spheres, parallel hyperplanes, and axis-parallel wedges.  The proof techniques rely on new Borsuk--Ulam type theorems on spheres and Stiefel manifolds.
\end{abstract}

\maketitle

\section{Introduction}

The ham sandwich problem is a classic result in topological combinatorics, attributed to Banach by Seinhaus \cite{Steinhaus1938}.  It says that \textit{for any $d$ mass distributions on $\rr^d$, there exists a hyperplane simultaneously bisecting all of them.}  A mass distribution $\mu$ is a probability measure absolutely continuous with respect to the Lebesgue measure, and a bisecting hyperplane is a hyperplane so that the two closed half-spaces it induces have the same $\mu$-measure.  The study of equitable partitions of mass distributions in Euclidean spaces exhibits the connections between combinatorics, topology, and mathematical economics \cites{RoldanPensado2022, Zivaljevic2017}.

In recent years, many existing mass partition results have been extended to the setting of mass assignments.  Instead of working with mass distributions on $\rr^k$, we continuously assign mass distributions to some $k$-dimensional subspaces of $\rr^d$.  The goal is to guarantee the existence of a particular $k$-dimensional subspace of $\rr^d$ in which we can find an equitable partition of the assigned mass distributions.  The freedom to choose the $k$-dimensional subspace often translates to improved results over the $k$-dimensional versions of mass partitions results (i.e., we are able to evenly split more masses simultaneously).  Consider the ham sandwich theorem for mass assignments, proven by Schnider \cite{Schnider:2020kk}.  It says that \textit{given $d$ mass assignments on the $k$-dimensional subspaces of $\rr^d$, there exists a $k$-dimensional subspace $L$ and a hyperplane $H$ of $L$ that simultaneously bisects all $d$ mass distributions assigned to $L$.}  Even though a direct application of the ham sandwich theorem would guarantee we can simultaneously bisect $k$ of these mass distributions, the additional freedom to choose $L$ allows us to bisect even more.  One can impose additional conditions on $L$, such as fixing $k-1$ of its directions in advance \cite{AxelrodFreed2022}.

In this paper, we prove mass assignment versions of classic mass partition results that use different shapes to split subspaces.  For instance, the first application of the polynomial ham sandwich theorem of Stone and Tukey is the following folklore result.

\begin{theorem}[Ham sandwich for spheres \cite{Stone:1942hu}]\label{thm:StoneTukey-old}
    Given $d+1$ mass distributions on $\rr^d$, there exists a sphere that contains exactly half of each mass distribution.
\end{theorem}

For the result above, we consider half-spaces as degenerate spheres of infinite radius.  Our first result is a mass assignment version of the Stone--Tukey sphere partition.

\begin{theorem}\label{thm:spheres-in-mass-assignments}
    Let $1 \le k \le d$ be integers.  Given $d+1$ mass assignments on the $k$-dimensional affine spaces of $\rr^d$, there exist a $k$-dimensional space $L$ of $\rr^d$ and a sphere on $L$ that contains exactly half of each mass distribution assigned to $L$.  Moreover, we can fix $k-1$ directions of $L$ in advance.
\end{theorem}

We also show that the number of mass assignments, $d+1$, is optimal for any $k$.  An easy to visualize example of this result is the following corollary.

\begin{corollary}\label{cor:lines}
    Let $L_1, L_2, L_3, L_4$ be four families of lines in $\rr^3$, and such that no line in the families is vertical and no two lines are parallel.  Then, there exists a vertical circle such that for each family $L_i$, at least half the lines go through the circle (i.e., the circle goes around the line) and at least half the lines do not.
\end{corollary}

A vertical circle is simply a circle contained in a vertical plane.  In the theorem above, if a line intersects the circle, we count it as going through through the circle and not going through the circle.  It may be possible that with some general position conditions and asking each family of lines to have an even number of lines we can also ask for each line not to intersect the circle.

The techniques used to prove mass partition results and their mass assignments analogues are topological.  For mass assignments, these methods range from the explicit computation of characteristic classes of associated vector bundles \cites{Schnider:2020kk, Blagojevic2023, Camarena2024}, properties of the Fadell--Husseini index \cite{Blagojevic2023a}, or applications of Borsuk--Ulam-type theorems for Stiefel manifolds \cite{AxelrodFreed2022}.  The last approach, using Borsuk--Ulam-type theorems for Stiefel manifolds, requires much simpler topological machinery.

The proof of \cref{thm:StoneTukey-old} boils down to a lifting argument and an application of the ham sandwich theorem in a higher dimension.  To prove \cref{thm:spheres-in-mass-assignments}, we combine similar lifting arguments with the parametrizations used to prove the ham sandwich theorem for mass assignments.  This reduces the problem to proving the existence of zeros of certain $(\zz_2)^{k+1}$-equivariant continuous maps with domain $S^d \times V_k(\rr^d)$, where $S^d$ is a $d$-dimensional sphere and $V_k(\rr^d)$ is a Stiefel manifold.  Even though similar results have been used to prove mass partition results \cite{Soberon2023}, the actions of $(\zz_2)^{k+1}$ we require are different than those in previous instances.

With similar methods, we also prove a mass assignment version of the Takahashi--Sober\'on theorem for mass partitions with two parallel hyperplanes.

\begin{theorem}\label{thm:goalposts}
    Let $1 \le k \le d$ be integers.  Given any $d+1$ mass assignments on the $k$-dimensional linear spaces of $\rr^d$, there exists a $k$-dimensional space $L$ of $\rr^d$ and two parallel hyperplanes of $L$ such that the region between the two parallel hyperplanes contains exactly half of each of the $d+1$ mass distributions assigned to $L$.
\end{theorem}

The case $k=d$ is the Takahashi--Sober\'on theorem.  An alternative, non-topological proof of the $k=d$ case was recently found by Hubard and Sober\'on \cite{Hubard2024}.  This result has also been extended to more parallel hyperplanes independently by Sadovek and Sober\'on, with very different methods \cite{Sadovek2024}.  

Another mass partition result whose mass assignment version is unknown is the existence of mass partitions by axis-parallel wedges in the plane \cite{Uno:2009wk}.  An axis-parallel wedge in $\rr^2$ is constructed by taking a point $p$ and two infinite rays starting from $p$, one vertical and one horizontal.  Note that the vertical ray could go into one of two possible directions, and the same holds for the horizontal ray.  Uno, Kawano, and Kano proved the following result:

\begin{theorem}[Uno, Kawano, Kano 2009 \cite{Uno:2009wk}]
    Given two mass distributions in the plane, there exists an axis-parallel wedge that bisects each mass distribution.
\end{theorem}

An alternative proof can was found by Karasev, Rold\'an-Pensado, and Sober\'on \cite{Karasev:2016cn}.  We show that the result above has a mass assignment version for vertical planes in $\rr^d$.  We define a vertical plane in $\rr^d$ as a two-dimensional subspace that contains the direction $e_d = (0,\dots,0,1)$.  Given a vertical plane $H$ in $\rr^d$, the definition of an axis-parallel wedge in $H$ is still clear, as the vertical direction is contained in $H$, and the horizontal direction is simply the direction orthogonal to $e_d$ in $H$.

\begin{theorem}\label{thm:axis-parallel-wedges-counterexample}
    For any $d$ mass assignments to vertical planes in $\rr^d$ there exists a vertical plane $H$ and an axis-parallel wedge in $H$ that bisects each of the $d$ measures assigned to $H$.\end{theorem}

\section{Stiefel manifolds and equivariant maps}

In this section we establish the results regarding equivariant maps required for our main results.  We denote by $V_k(\rr^d)$ the Siefel manifold of orthonormal $k$-frames in $\rr^d$.  Namely,
\[
V_k(\rr^d) = \{(v_1, \dots, ,v_k): v_1, \dots, v_k \in \rr^d \mbox{ are orthonormal}\}.
\]

Note that $V_1 (\rr^d) = S^{d-1}$, a $(d-1)$-dimensional sphere, and that the dimension of $V_{k}(\rr^d)$ is equal to $(d-1) + (d-2) + \dots + (d-k)$.  We consider the group $\zz_2 = \{+1,-1\}$ with multiplication.  There is a free action of $(\zz_2)^k$ on $V_k(\rr^d)$ given by
\[
(\lambda_1, \dots, \lambda_k)\cdot (v_1, \dots, v_k) = (\lambda_1 v_1, \dots, \lambda_kv_k).
\]
To simplify notation later in the manuscript, we will work with $V_{d-k}(\rr^d)$ instead of $V_k(\rr^d)$.  The product $S^d \times V_{d-k}(\rr^d)$ has a free action of $(\zz_2)^{d-k+1}$ by taking the product of the action of $\zz_2$ on $S^d$ and the action of $(\zz_2)^{d-k}$ on $V_{d-k}(\rr^d)$.

Another space with an action of $(\zz_2)^{d-k+1}$ is $R = \rr^{d} \times \rr^{d-1} \times \dots \times \rr^k$.  Note that $R$ has the same dimension as $S^d \times V_{d-k}(\rr^d)$.  To describe the action of $(\zz_2)^{d-k+1}$ on $R$, we describe the action of its generators.  For $i=0,\dots, d-k$, let $m_i = (\underbrace{1,\dots, 1,}_{i}-1,\underbrace{1,\dots, 1}_{d-k-i})$.  Let $(x_0, x_1,\dots, x_{d-k}) \in R$.  Then, for $i=1,\dots,d-k$, the action $m_i\cdot (x_0, x_1,\dots, x_{d-k})$ changes the sign of every coordinate of $x_i$ and nothing else.  Finally, $m_0\cdot (x_0, x_1,\dots, x_{d-k})$ changes the sign of every coordinate of $x_0$ and the first coordinate of each of $x_1, \dots, x_{d-k}$.

The main topological tool we need is the following.

\begin{theorem}\label{thm:weird-action-equivariant}
    Let $d \ge k \ge 0$ be integers.  Let $f: S^d \times V_{d-k}(\rr^d) \to R$ be a continuous $(\zz_2)^{d-k+1}$-equivariant map, with the actions described above.  Then, $f$ must have a zero.
\end{theorem}

For the case $k=d$, we consider $\rr^0 = \{0\}$.  \cref{thm:weird-action-equivariant} is similar to the one used by Sober\'on and Takahashi to prove extensions of the ham sandwich theorem using pairs of parallel hyperplanes \cite{Soberon2023}, but in their proof methods the action of $m_0$ did not affect any coordinates of $x_i$ for $i \ge 1$.  Yet, we can use \cref{thm:weird-action-equivariant} to prove the Sober\'on--Takahashi theorem, which would simplify the original proof slightly (one additional trick used by Sober\'on and Takahashi is no longer needed).

We can prove \cref{thm:weird-action-equivariant} using one of several different approaches.  One is to verify the conditions of the existence of zeros of similar maps by Chan, Chen, Frick, and Hull \cite{Chan2020}.  Another, using the fact that the domain and co-domain of $f$ can be approximated by piecewise-linear manifolds of the same dimension, is to apply Musin's general theorem for the existence of zeros for equivariant maps between such spaces directly \cite{Mus12}.  A third would be to compute the Fadell--Husseini index of the two $(\zz_2)^{d-k+1}$-spaces \cite{Fadell:1988tm}.  Fadell and Husseini computed the indices for Stiefel manifolds and sphere, and the effect of the modified action on $R$ would be the only additional detail to revise.

For the sake of completeness, we include a proof similar to the one by Musin, based on B\'ar\'any's proof of the Borsuk--Ulam theorem \cite{Barany1980}.  Similar modifications have been used for $SO(3)$-equivariant maps by Fradelizi et al. \cite{Fradelizi2022}.

\begin{proof}
    We begin by constructing a particular $(\zz_2)^{d-k+1}$-equivariant map 
    \begin{align*}
        g:S^d \times V_{d-k}(\rr^d) & \to R \\
        (w, v_1, \dots, v_{d-k}) & \mapsto (x_0, x_1,\dots, x_{d-k})
    \end{align*}
    In the statement above, $w \in S^d \subset \rr^{d+1}$ and $(v_1, \dots, v_{d-k}) \in V_{d-k}(\rr^d)$.  Let $v_{i,j}$ be the $j$-th coordinate of $v_i$, and $w_j$ be the $j$-th coordinate of $w$.

    Then, we define 

    \begin{align*}
        x_0 = \begin{bmatrix}
            w_1 \\ w_2 \\ \vdots \\ w_{d}
        \end{bmatrix} \in \rr^d, \qquad v_i = \begin{bmatrix}
            w_{d+1} \cdot v_{i,1} \\ v_{i,2} \\ \vdots \\ v_{i,d-i}
        \end{bmatrix}\in \rr^{d-i} \mbox{ for }i=1,\dots,d-k.
    \end{align*}

    This map is $(\zz_2)^{d-k+1}$-equivariant by construction. Moreover, if $g(w,v_1,\dots, v_{d-k})=0$, then $w = \pm e_{d+1}$, the north or south pole of $S^d$.  This means that $w_{d+1}= \pm 1$.  Then, an inductive argument shows that $v_i = \pm e_{d-i}$, the $(d-i)$-th canonical vector of $\rr^{d}$ or its negative.  Overall, there is exactly one $(\zz_2)^{d-k+1}$ orbit of zeros in $S^d \times V_{d-k}(\rr^d)$.

    Now consider the map

    \begin{align*}
        H:S^d \times V_{d-k}(\rr^d)\times[0,1] & \to R \\
        (w, v_1, \dots, v_{d-k},t) & \mapsto t \cdot f(w, v_1, \dots, v_{d-k}) + (1-t)\cdot g(w, v_1, \dots, v_{d-k}).
    \end{align*}

    If we denote by $f_t (\cdot) = H(\cdot, t)$, we can see that $f_t$ is a $(\zz_2)^{d-k+1}$-equivariant map, $f_0 = g$, and $f_1 = f$, our target map.  If $H$ happens to be smooth enough, then $H^{-1}(0)$ will be a $1$-dimensional manifold.  This means that its connected components will be segments or cycles.  The components which are segments must have their enpoints in the boundary of $S^d \times V_{d-k}(\rr^d)\times[0,1]$, which is $\big(S^d \times V_{d-k}(\rr^d)\times\{0\}\big) \bigcup \big(S^d \times V_{d-k}(\rr^d)\times\{1\}\big)$.  In other words, they connect zeros of $g$ with either zeros of $g$ or zeros of $f$.  In the latter case, this would imply the existence of zeros of $f$, and we would be done.  Let us assume for the contradiction that these segments connect zeros of $g$ with zeros of $g$. 
    
    The action of $(\zz_2)^{d-k+1}$-equivariant map  acts on the connected components of $H^{-1}(0)$, and therefore must send the connected components which are segments to segments.  Suppose $p_0$ is a zero of $g$ that is connected via a segment $I$ in $H^{-1}(0)$ to $p_1$, another zero of $g$.  Since there is exactly one orbit of zeros of $g$, there must exist $\lambda \in (\zz_2)^{d-k+1}$ such that $\lambda \cdot p_0 = p_1$.  In particular, $\lambda (I) = I$.  Multiplication by $\lambda$ is a continuous function, so by the intermediate value theorem $\lambda|_{I}$ must have a fixed point.  This contradicts the fact that the action of $(\zz_2)^{d-k+1}$ on $S^d \times V_{d-k}(\rr^d)\times[0,1]$ is free.

    If $H$ is not smooth enough, we can still approximate it by sufficiently smooth maps $H'$ which agree with $H$ on $V_{d-k}(\rr^d)\times\{0\}$ by applying Thom's transversality theorem \citelist{\cite{thom1954quelques} \cite{guillemin2010differential}*{pp 68-69}}.  The compactness of $S^d \times V_{d-k}(\rr^d)\times[0,1]$ makes the approximation argument give a zero of the original function $f$.
\end{proof}

\section{Spheres and mass assignments}

In this section we use \cref{thm:weird-action-equivariant} to prove \cref{thm:spheres-in-mass-assignments}.  We also show that \cref{thm:spheres-in-mass-assignments} is optimal.

\begin{proof}[Proof of \cref{thm:spheres-in-mass-assignments}]
    To show that we can impose any $k-1$ directions on $H$, we will impose that $H$ contains each of the directions $e_1, \dots, e_{k-1}$.  The methods below work for any $k-1$ arbitrary vectors.
    
    We first construct a $(\zz_2)^{d-k+1}$-equivariant function $f: S^d \times V_{d-k}(\rr^d) \to R = \rr^d \times \rr^{d-1} \times \dots \times \rr^{k}$.  The goal will be that the zeros of this function yield pairs $(L,S)$ where $L$ is a $k$-dimensional subspace of $\rr^d$ and $S \subset L$ is a sphere that bisects each of the mass distributions assigned to $S$.  Let $\mu_0, \dots, \mu_d$ be the $d+1$ mass assignments to the $k$-dimensional subspaces of $L$.  Given a subspace $L$ we will denote by $\mu_0^L,\dots, \mu_d^L$ the mass distributions assigned to $L$.

    Let $i: \rr^d \hookrightarrow \rr^{d+1}$ be the embedding that appends a coordinate $1$; so $i(x) = (x,1)$.  Let $V = i(\rr^d)$ be the embedding of $\rr^d$ into $\rr^{d+1}$.  Let $\sigma$ be the inversion in $\rr^{d+1}$ of radius $1$ centered at the origin.  Namely, $\sigma(x) = x/\|x\|^2$ for $x \neq 0$.  Note that $\sigma$ sends affine hyperplanes to spheres that contain $0 \in \rr^{d+1}$.  Additionally, $\sigma \circ \sigma$ is the identity.  We include an illustration in \cref{fig:inversion-1}.  This means that if $H$ is an affine hyperplane in $\rr^{d+1}$, then $\sigma (H)$ is a sphere in $\rr^{d+1}$.  In other words, $\sigma (H )\cap V$ will be a sphere in $V=i(\rr^d)$.  This construction is represented in \cref{fig:inversion-2}.

\begin{figure}
    \centering
    \includegraphics[width=0.7\linewidth]{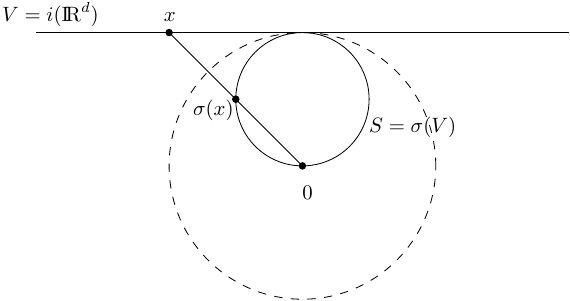}
    \caption{An exemplification of the construction for $S$.  We embed $\rr^d$ into $\rr^{d+1}$ and then use the inversion $\sigma$.  The dashed sphere is the set of fixed points by $\sigma$.}
    \label{fig:inversion-1}
\end{figure}

\begin{figure}
    \centering
    \includegraphics[width=0.7\linewidth]{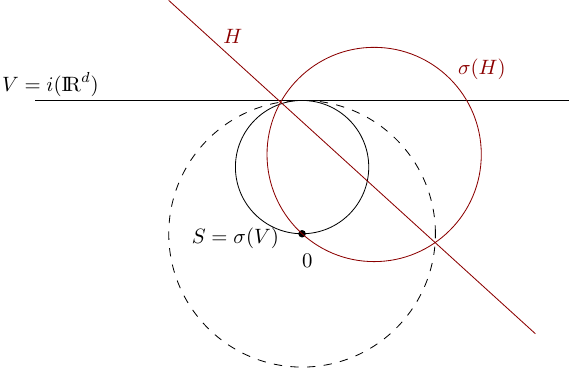}
    \caption{If we have a hyperplane $H$ halving the measures induced in $S$, then $\sigma(H)$ is a sphere so that $\sigma(H) \cap V$ halves the measures in $V$.  If $0\in H$, then $\sigma(H) = H$, which we count as a degenerate sphere.}
    \label{fig:inversion-2}
\end{figure}

    Denote by $S = \sigma(V)$ the sphere corresponding to $\rr^d$ after the inversion $\sigma$.  $S$ is a sphere of radius $1/2$ with center at $i(0)/2=(0,\dots, 0,1/2)\in \rr^{d+1}$.   Given a $k$-dimensional subspace $L$ of $\rr^d$, consider $L' = \sigma \circ i (L)$ the corresponding subset of $S$.  A measure $\mu$ on $L$ induces a measure on $L'$ and therefore a measure $\mu'$ on $\rr^{d+1}$.

    Before we continue, let us explain the main strategy behind our construction.  We will use the space $S^d \times V_{d-k}(\rr^d)$ to parametrize pairs $(H,L)$, where $H$ is a hyperplane in $\rr^{d+1}$ and $k$-dimensional subspace $L$ of $\rr^d$.  When testing a pair $(H,L)$, we will lift $L$ to $\rr^{d+1}$ by looking at the set $L'$, and obtain $d+1$ measures on $\rr^{d+1}$ using the mass distributions assigned to $L'$.  We will be able to assume without loss of generality that $H$ halves one of those measures.  We will then test if $H$ halves the rest of these measures.  If it does, then $i^{-1}(\sigma(L' \cap H)$ will be a the sphere in $L$ we are looking for.  One delicate point in the proof is that we have to make sure that $H$ does not contain $L'$.

    Given $(w,v_1,\dots, v_{d-k}) \in S^d \times V_{d-k}(\rr^d)$, let $L$ be the orthogonal complement of $\operatorname{span}\{v_1,\dots, v_{d-k}\}$.  Our candidate of $H$ will be orthogonal to $w$.  We describe below how we choose the correct translate of $w^{\perp}$.  We also consider
    \[
    U = \operatorname{span}\{(v_1,0), \dots, (v_{d-k},0)\},
    \]
    which we will use to make sure that $H$ does not contain $L'$.  Let $\nu_0^L, \dots, \nu_{d}^L$ be the $d+1$ measures in $\rr^{d+1}$ obtained from lifting each of $\mu_{0}^L, \dots, \mu_{d}^L$ to $\sigma \circ i (L) \subset S \subset \rr^{d+1}$. 
    Among all translates of $w^{\perp}$, let $H$ be the one such that its two closed half-spaces have the same $\nu_0^L$-measure.  It is possible that there is an interval of translates that satisfy this condition.  If this is the case, we pick the translate in the midpoint of this interval.

    Note that $L' \subset H$ if and only if $w \in U$.  Moreover, if $\langle w, (v_i,0)\rangle = 0$ for all $i$, then $w \not\in U$, which implies that $L' \not\subset H$.  Now we are ready to construct the map
    \begin{align*}
    f: S^d \times V_{d-k}(\rr^d)& \to \rr^{d} \times \dots \times \rr^{k} \\
    (w,v_1,\dots, v_{d-k}) & \mapsto (x_0, \dots, x_{d-k}).
    \end{align*}

    Let $H^+$ be the closed half-space of $H$ in the direction of $w$, and $H^{-}$ be the closed half-space of $H$ in the direction of $-w$.

    We first define
    \[
x_0 = \operatorname{dist}(w,U)\begin{bmatrix}
    \nu_1^L(H^+) - \nu_1^L(H^-) \\
    \vdots \\
    \nu_d^L(H^+) - \nu_d^L(H^-)
\end{bmatrix} \in \rr^d
    \]
    and for $i = 1, \dots, d-k$, we define
    \[
    x_i = \begin{bmatrix}
        \langle w, (v_i,0)\rangle \\
        \langle e_1 , v_i \rangle \\
        \langle e_{2} , v_i \rangle \\
        \vdots \\
        \langle e_{k-1}, v_i \rangle \\
        0 \\
        \vdots \\
        0
    \end{bmatrix} \in \rr^{d-i}
    \]
    We are fixing the first $k$ coordinates of $x_i$, and then filling the rest with zero.  The function $f$ is $(\zz_2)^{d-k+1}$-equivariant by construction.  The use of $\operatorname{dist}(w,U)$ in the definition of $x_0$ is to make sure the function is continuous even when $w \in U$.

    By \cref{thm:weird-action-equivariant}, the map $f$ must have a zero. 
 Suppose that $(w,v_1,\dots, v_{d-k})$ is a zero of $f$.  In this case, since $\langle w, (v_i,0)\rangle = 0$ for all $i$, we have that $w \not\in U$.  Therefore, $K=i^{-1} (\sigma (H \cap L'))$ will be a sphere in $L$.  Since $x_0 = 0$ and $\operatorname{dist}(w,U)>0$, we also have $ \nu_j(H^+) = \nu_j(H^-)$ for $j=1,\dots, d$, which means that $K$ will bisect each of the $d+1$ masses assigned to $H$ (the first mass distribution is bisected by construction).  Finally, since $\langle e_l, v_i \rangle = 0$ for $l=1,\dots, k-1$ and $i = 1,\dots, d-k$, we have that $L$ will contain the directions $e_1, \dots, e_{k-1}$.
\end{proof}

Now we show that \cref{thm:spheres-in-mass-assignments} is optimal, even if we ignore the condition of prescribing $k-1$ directions for $L$.

\begin{claim}\label{claim:optimality}
    Let $1 \le k \le d$ be integers.  There exists a set of $d+2$ mass assignments on the $k$-dimensional subspaces of $\rr^d$ for which there does not exist a $k$-dimensional space $L$ and a sphere in $L$ that bisects each of the mass distributions assigned to $L$.
\end{claim}

\begin{proof}
    
Let $p_0, \dots, p_{d+1}$ be points in $\rr^d$ such that $p_1, \dots, p_{d+1}$ form a regular simplex and $p_0$ is the barycenter of said simplex.  Let $r>0$ be small enough so that 
    $$B_r^d(p_0)\subset \Conv(p_1,\dots, p_{d+1}).$$
Then for any $k$-dimensional linear subspace $L$ of $\rr^d$ with projection $\pi_L: \R^d\to  L$, we have
    $$\pi_L(B_r^d(p_0)) \subset \pi_L(\Conv(p_1,\dots, p_{d+1})) = \Conv(\pi_L(p_1),\dots \pi_L(p_{d+1})).$$
For $i=0,\dots,d+1$, let $\nu_i$ be the uniform measure on the ball centered at $p_i$ with radius $r$.  For any $k$-dimensional subspace $L$ of $\rr^d$, let $\mu^L_i$ be the projection of $\nu_i$ onto $L$.

If a ($k$-dimensional) ball $B^k$ in $L$ contains at least half of $\mu^L_i$, then it must contain $\pi_L(p_i)$. If this were to happen for each $i=1,\dots, d+1$, then we would have
    $$B^k\supset \Conv(\pi_L(p_1),\dots \pi_L(p_{d+1})) \supset \pi_L(B_r^d(p_0)).$$
This would imply that $\mu_0^L (B^k) = \mu_0^L (L)$, which is much more than half of $\mu_0^L$. We conclude that no ball in $L$ can contain exactly half of each $\mu_i^L.$
\end{proof}

\begin{proof}[Proof of \cref{cor:lines}]
A family of lines in $\rr^3$ intuitively induces a mass assignment on the $2$-dimensional subspaces of $\rr^3$.  Given a plane $H \subset \rr^3$ and a Borel set $A \subset H$, define $\mu_H(A)$ the number of lines that intersect $A$.  We have the issue that lines can be parallel to a plane, so we have to be careful with this process.  We can avoid these issues using the same ideas as \cite{AxelrodFreed2022}.

    A finite family of lines is a discrete measure on $A_1(\rr^3)$, which is the affine Grassmanian of $1$-dimensional lines in $\rr^3$.  Note that the embedding $x \mapsto (x,1)$ from $\rr^3$ to $\rr^4$ shows that we can consider $A_1(\rr^3)$ as a subset of $G_2(\rr^4)$, the Grassmanian of $2$-dimensional linear subspaces of $\rr^4$.  We will say that a measure on $A_1(\rr^3)$ is smooth if its lift to $G_2(\rr^4)$ is a measure that is absolutely continuous with respect to the Haar measure on $G_2(\rr^4)$.

    Now, a smooth measure $\mu$ on $A_1(\rr^3)$ induces a mass assignment on the $2$-dimensional linear subspaces of $\rr^2$.  Given a $2$-dimensional subspace $H$ and a Borel set $A \subset H$, we define

    \[
    \mu_H(A) = \mu\left(\{\ell \in A_1(\rr^3): \ell \cap A \neq \emptyset \}\right).
    \]

    Since the set of lines that do not intersect $H$ has $\mu$-measure $0$, this is a mass assignment.  Therefore, given four smooth measures $\mu_1, \mu_2, \mu_3, \mu_4$ on $A_1(\rr^3)$ we can apply \cref{thm:spheres-in-mass-assignments} and obtain a vertical plane and a circle $\Gamma$ in that plane such that the $\mu_i$-measure of the set of lines that intersect the closed flat disc induced by $\Gamma$ is half of $\mu_i(A_1(\rr^3))$.  Finally, since $G_2(\rr^4)$ is a metric space, we can approximate the discrete measures induced by each family of lines by smooth measures.  The compactness of $G_2(\rr^4)$ shows that this approximation argument leads to the circle we want in $\rr^3$.

\end{proof}

\section{Mass assignments and parallel hyperplanes}

We expand on the ideas from the previous section to prove \cref{thm:goalposts}.  The case $k=d$ provides a new proof of the Sober\'on--Takahashi theorem for mass partitions with parallel hyperplanes \cite{Soberon2023}.

\begin{proof}
    We will use the set $S^d \times V_d(\rr^d)$ to parametrize our set of candidate partitions, regardless of the value of $k$.  We also want to find a map $f:S^d \times V_d(\rr^d) \to \rr^d \times \rr^{d-1}\times \dots \times \rr^0$ that checks if we have a correct partition.  Let $(w,v_1,\dots, v_d)$ be an element of $S^d \times V_d(\rr^d)$.  As in the previous section, our subspace $L$ will be the orthogonal complement of $v_1, \dots, v_{d-k}$.  However, it will be useful to have a basis of $L$, which is why we use the last $k$ elements of the $d$-frame: 
    \[
    L = \operatorname{span}\{v_{d-k+1}, \dots, v_d\} 
    \]
    The way we lift $L$ to $\rr^{d+1}$ is now given by the embedding
    \begin{align*}
        i: L & \to \rr^{d+1} \\
        \sum_{j=d-k+1}^d a_j v_j & \mapsto \left(\sum_{j=d-k+1}^{d} a_j (v_j,0)\right) + a_d^2 e_{d+1}.
    \end{align*}
    Intuitively, this map wraps the space $L$ upwards like a parabola preserving the direction $v_d$.

Given $d+1$ mass assignments on $k$-dimensional mass assignments, let $\nu_0^L, \dots, \nu_{d}^L$ be the $d+1$ measures on $\rr^{d+1}$ obtained by lifting the mass distributions assigned to $L$ via $i$.  Again, we pick a halfspace $H$ orthogonal to $w$ that bisects $\nu_0^{L}$, which can be chosen canonically halfway through all such possible hyperplanes.

It suffices to verify that $\langle w, (v_i,0)\rangle = 0$ for $i = 1,\dots , d-k$ so that $i(L) \not \subset H$, which means that $i(L)\cap H$ will be a set of co-dimension $1$ in $i(L)$.  Moreover, if we also verify that $\langle w, (v_i,0) \rangle = 0$ for $i=d-k+1, \dots, d-1$, then $H \cap i(L)$ projects back to $L$ as the union of two parallel hyperplanes, orthogonal to $v_d$.  This is because in the projections of the $2$-dimensional space orthogonal to $(v_i,0)$ for $i=1,\dots,d-1$, the image of $L'$ is a parabola and the image of $H$ is a line.  Since a line can intersect a parabola in at most two points, the translates of $\operatorname{span}\{(v_{d-k+1},0), \dots, (v_{d-1},0)\}$ through these two points will give us $H \cap L'$, which project back to $L$ onto two parallel hyperplanes.  We are ready to define $f(w,v_1,\dots, v_d)=(x_0,\dots, x_d)$ by
    \begin{align*}
x_0 & = \operatorname{dist}(w,U)\begin{bmatrix}
    \nu_1^L(H^+) - \nu_1^L(H^-) \\
    \vdots \\
    \nu_d^L(H^+) - \nu_d^L(H^-)
\end{bmatrix} \in \rr^d & \\
    x_i & = \begin{bmatrix}
        \langle w, (v_i,0)\rangle \\
        \langle e_1 , v_i \rangle \\
        \langle e_{2} , v_i \rangle \\
        \vdots \\
        \langle e_{k-1}, v_i \rangle \\
        0 \\
        \vdots \\
        0
    \end{bmatrix} \in \rr^{d-i} & \mbox{for }i=1,\dots,d-k \\
    x_i &= \begin{bmatrix}
        \langle w, (v_i,0)\rangle \\
        0 \\
        \vdots \\
        0
    \end{bmatrix} \in \rr^{d-i} & \mbox{for }i=d-k+1,\dots,d-1 \\
    x_d & = \begin{bmatrix}
        0
    \end{bmatrix} \in \rr^0
    \end{align*}

This map is continuous and equivariant, so it must have a zero by \cref{thm:weird-action-equivariant}.  A zero must satisfy $i(L) \not\subset H$.  This implies that $H$ bisects each of the measures $\nu_0^L, \dots, \nu_d^L$, and $i^{-1}(L'\cap H)$ will be the union of two parallel hyperplanes in $L$ such that the region between them contains exactly half of each of the $d+1$ mass distributions assigned to $L$.  This construction also ensures that $L$ contains each of the directions $e_1, \dots, e_{k-1}$.

    \end{proof}

    \cref{claim:optimality} also applies for \cref{thm:goalposts}.  This is because the only property of balls used in the proof of \cref{claim:optimality} is that they are convex.  Since the region between two parallel hyperplanes is also convex, then the same construction shows optimality of \cref{thm:goalposts}.

\section{Axis-parallel wedges}

Let us prove \cref{thm:axis-parallel-wedges-counterexample}.  To do this, we first introduce some useful constructions in $\rr^d$.

We say that an axis-parallel wedge $W$ in $\rr^2$ is a down-wedge if its vertical segment points down.  We consider vertical lines and horizontal lines as degenerate cases of down-wedges.

\begin{lemma}\label{lemma:unique-down-wedge}
    Let $\mu$ be a mass distribution on $\rr^2$ that has positive measure for every open set of $\rr^d$ and let $t$ be a real number.  Then there exists a unique down-wedge $W$ whose vertical ray is contained in the line $x=t$ that bisects $\mu_1$. 
\end{lemma}

\begin{proof}
    If the line $x=t$ bisects $\mu$, this is the unique down-wedge that bisects $\mu$.  Otherwise, one of the two sides of this line has $\mu$-measure greater than $\mu(\rr^d)/2$.  Let $X_t(y)$ the down-wedge with vertex $(t,y)$ and the horizontal ray pointing to the side of the line $x=t$ with larger $\mu$-measure.  The function $t \mapsto \mu (X_t(y))$ continuous.  Note that $\lim_{y \to \infty} \mu X_t(y) > 1/2$ by construction (it is the measure of the half-plane defined by $x=t$ with larger $\mu$-measure), and $\lim_{t\to -\infty} = 0$.  Therefore, by the intermediate value theorem there must be a value of $y$ such that $\mu(X_t(y))=1/2$.  This value is unique since the function $t \mapsto \mu (X_t(y))$ is strictly increasing.
\end{proof}

We introduce specific notation for this down-wedge

\begin{definition}\label{def:down-wedge-name}
Given a mass distribution $\mu_d$ as above and $t\in\rr$, we denote by $A(\mu,t)$ the convex region of the unique down-wedge from \cref{lemma:unique-down-wedge} if $x=t$ is not the halving vertical line for $\mu$.  If $x=t$ is the halving line, we denote by $A(t)$ the left closed half-plane of $x=t$.  If $\ell$ is the halving horizontal line for $\mu$, we can also define $A(\mu,\infty)$ as the bottom half-plane of $\mu$.  This way, if $\nu$ is any other measure we have that $\lim_{t \to \infty} \nu(A(\mu,t)) = \nu (A(\mu,\infty))$.   We denote by $B(\mu,t)$ the other closed side of this downwedge.    
\end{definition}

\begin{figure}
    \centering
    \includegraphics[width=0.5\linewidth]{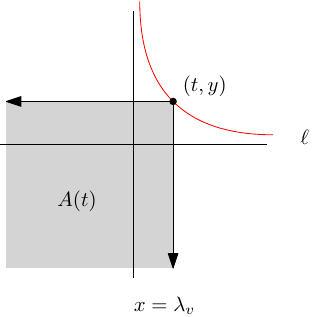}
    \caption{Given a measure $\mu$ with a vertical halving line at $x=\lambda_v$ and a horizontal halving line $\ell$, the unique down-wedge $A(t)$ with vertex satisfying $x=t$ is shown in the figure for some $t > \lambda_v$.  Notice that as $t\to \lambda_v$, the set $A(t)$ will approach the left side of $x=\lambda_v$, and as $t \to \infty$ the set $A(t)$ will approach the bottom half-plane defined by $\ell$. }
    \label{fig:down-wedge}
\end{figure}

An illustration of our construction is shown in \cref{fig:down-wedge}.  Now we are ready to prove \cref{thm:axis-parallel-wedges-counterexample}

\begin{proof}
    Assume we  have mass assignments $\mu_1,\dots, \mu_d$ on the vertical planes in $\rr^d$.  Assume that for each vertical plane $H$, the mass distribution assigned $\mu_d^H$ satisfies the conditions of \cref{lemma:unique-down-wedge}.  We address this additional assumption at the end of the proof.

    Let $S$ be the $(d-2)$-dimensional sphere orthogonal to $e_d$ in $\rr^d$.  We denote by $H_v$ the vertical plane spanned by $e_d$ and $v$, equipped with the ordered basis $(v,e_d)$.  For $v \in S$, let $\lambda_v$ be the value such that the line $x=\lambda_v$ in $H_v$ is the vertical halving line to $\mu^{H_v}_d$ in $H_v$.  For $t \in [\lambda_v,\infty]$, we define the sets $A^v(t), B^v(t)$ in $H_v$ from \cref{def:down-wedge-name} for $\mu^{H_v}_d$.

    Consider the map

    \begin{align*}
        f: S^{d-2}\times [0,1] & \to \rr^{d-1} \\
        (v,t) & \mapsto \begin{cases}
            \begin{bmatrix}
                \mu^{H_v}_1(A^v(\lambda_v -1 + 1/t))-\mu^{H_v}_1(B^v(\lambda_v -1 + 1/t)) \\
                \vdots \\
                \mu^{H_v}_{d-1}(A^v(\lambda_v -1 + 1/t))-\mu^{H_v}_{d-1}(B^v(\lambda_v -1 + 1/t))
            \end{bmatrix} & \mbox{ if }t > 0 \\
            \begin{bmatrix}
                \mu^{H_v}_1(A^v(\infty))-\mu^{H_v}_1(B^v(\infty)) \\
                \vdots \\
                \mu^{H_v}_{d-1}(A^v(\infty))-\mu^{H_v}_{d-1}(B^v(\infty))
            \end{bmatrix} & \mbox{ if } t=0
        \end{cases}
    \end{align*}

The sets $A(\lambda_v -1 +1/t)$ and $B(\lambda_v -1 +1/t)$ are two sides of a particular down-wedge that bisects $\mu_d^{H_v}$.  When $t=1$, they are the two half-plances defined by the vertical halving line of $\mu_d^{H_v}$.  When $t=\infty$, they are the two half-planes defined by the horizontal line of $\mu_d^{H_v}$.  If $f$ has a zero, then we have a vertical plane and a down-wedge bisecting all mass distributions on it.

Assume in search of a contradiction that the function $f$ has no zeros.  Then, we can consider the function

\begin{align*}
    g: S^{d-2}\times [0,1] & \to S^{d-2} \\
    (v,t) & \mapsto \frac{f(v,t)}{\|f(v,t)\|}
\end{align*}

For $t \in [0,1]$ consider the map $g_t: S^{d-2} \to S^{d-2}$ defined by $v \mapsto g(v,t)$.  The map $g$ provides a homotopy between $g_0$ and $g_1$.  Note that since $A^v(\lambda_v) = B^{-v}(\lambda_{-v})$ and $B^v(\lambda_v) = A^{-v}(\lambda_{-v})$ (as we only switch two sides of a vertical half-plane), we have that $g_1(-v) = - g_1(v)$, so $g_1$ must have odd degree.  However, since $A^v(\infty) = A^{-v}(\infty)$ and $B^v(\infty) = B^{-v}(\infty)$, as they are both defined as the bottom and top half-planes of the same horizontal line.  This means $g_0(v) = g_0(-v)$.  Therefore, $g_0$ must have even degree.  This contradicts the fact that the degree of a map is invariant under homotopies, so the conclusion of the theorem must hold.

The assumption on $\mu_d$ does not cause a loss of generality, as they can be addressed with the following standard compactness argument.  If $\mu_d$ does not satisfy those conditions, we can consider a mass assignment $\nu_d$ that does and take for some $\varepsilon > 0$ the assignment $\mu_{d,\varepsilon} = \frac{1}{1+\varepsilon}(\mu_d + \varepsilon \nu_d)$  In other words, for any vertical plane $H$ and any measurable set $A \subset H$ we have $\mu^H_{d,\varepsilon}(A) = \frac{1}{1+\varepsilon}(\mu^H_d(A) + \varepsilon \nu^H_d(A))$.  We apply the proof above and obtain the existence of a vertical plane $H_{\varepsilon}$ and a down-wedge $W_{\varepsilon}$ that bisects each of $\mu^H_{1}, \dots , \mu^H_{d-1}, \mu^H_{d,\varepsilon}$.  The set of vertical planes is compact, and the set of down-wedges that bisect a mass assignment is also compact if we consider degenerate cases (horizontal and vertical lines, the entire plane, and the empty set).  Therefore, by taking $\varepsilon \to 0$ we can construct a sequence of vertical planes and down-wedges that converge, which ultimately also bisects $\mu_d$.
\end{proof}


\end{document}